\documentclass[11pt]{amsart}

\usepackage{graphicx}
\usepackage{amssymb}
\usepackage{amsmath}
\usepackage{amsthm}
\calclayout

\newcommand\restr[2]{{
  \left.\kern-\nulldelimiterspace
  #1
  \vphantom{\big|} 
  \right|_{#2} 
  }}

\newtheorem{theorem}{Theorem}[section]
\newtheorem{lemma}[theorem]{Lemma}

\newtheorem{definition}[theorem]{Definition}
\newtheorem{proposition}[theorem]{Proposition}
\newtheorem*{thma}{Theorem 1.1}
\newtheorem*{corb}{Corollary 1.2}

\begin{document}
\title{Maximal Amenability with Asymptotic Orthogonality in Amalgamated Free Products}
\author{Brian Leary}
\address{Department of Mathematics, West Virginia University Institute of Technology, 410 Neville Street, Beckley, WV 25801}
\email{Brian.Leary1@mail.wvu.edu}
\begin{abstract} \noindent We investigate the use of Popa's asymptotic orthogonality to establish maximal amenability for amalgamated free product von Neumann algebras.  Although new techniques have recently been developed to consider amalgamated free products, we find that the technique developed by Popa in the 1980s can be used to demonstrate maximal amenability for a certain family of amalgamated free product von Neumann algebras.
\end{abstract}
\maketitle
\section{Introduction}
Our goal in this paper is to establish maximal amenability results in certain amalgamated free product von Neumann algebras by using Popa's asymptotic orthogonality method.

In the origins of the subject in the 1930s and 1940s, Murray and von Neumann gave the basic definitions of factoriality, type decomposition, and other isomorphism invariant properties in what are now known as von Neumann algebras, and they constructed the first examples.  One such construction was the approximately finite dimensional $\text{II}_1$ factor $\mathcal{R}$, which can be realized as the tensor product (with respect to the normalized trace) 
of countably infinitely many copies of the algebra of $2\times 2$ matrices over the complex numbers.  They were also able to prove that, up to isomorphism, $\mathcal{R}$ is the unique approximately finite dimensional factor of type $\text{II}_1$.  Moreover, they showed that every infinite dimensional factor contains a copy of $\mathcal{R}$.  Later, Connes \cite{C} was able to prove that the property of a factor $M\subset B(\mathcal{H})$ being approximately finite dimensional is equivalent to the amenability of $M$, which is the existence of an $M$-central state on $B(\mathcal{H})$ that extends the trace on $M$, and is also equivalent to the injectivity of $M$, which is the existence of a conditional expectation from $B(\mathcal{H})$ onto $M$.

In Kadison's 1967 list of problems on von Neumann algebras \cite{Kad}, he asked whether every self-adjoint operator in an arbitrary $\text{II}_1$ factor $M$ can be embedded into some approximately finite dimensional subfactor of $M$, or equivalently, whether every separable abelian von Neumann subalgebra of a $\text{II}_1$ factor $M$ could be embedded into some approximately finite dimensional subfactor of $M$.  In 1983, Popa \cite{P1} provided a negative answer to the problem by constructing an abelian subalgebra of a $II_1$ factor that is a maximal amenable subalgebra, and hence a maximal approximately finite dimensional subalgebra by Connes' theorem.  Specifically,  he proved that the abelian von Neumann subalgebra $M_a$ of $L(\mathbb{F}_n)$, generated by one of the $n$ canonical generators $a$ of $\mathbb{F}_n$, is maximal amenable in $L(\mathbb{F}_n)$.  His method involved the analysis of $M_a$-central sequences in $L(\mathbb{F}_n)$ through the ``asymptotic orthogonality property."  To be precise, he showed that for any free ultrafilter $\omega$ on $\mathbb{N}$, any $x\in M_a^\prime \cap (L(\mathbb{F}_n)^\omega \ominus M_a^\omega)$, and any $y_1, y_2\in L(\mathbb{F}_n)\ominus M_a$, we have that $y_1x$ and $xy_2$ are perpendicular in $L^2(L(\mathbb{F}_n)^\omega)$.  Then using the strong mixingness of $M_a$ inside $L(\mathbb{F}_n)$, he showed that $L(\mathbb{F}_n)^\prime\cap M_a$ had a non-zero atomic part, and from this he was able to conclude maximal amenability.

In the last decade, Popa's asymptotic orthogonality technique has been used to obtain several more results about maximal amenable subalgebras.  Cameron, Fang, Ravichandran, and White proved in \cite{CFRW} that the Laplacian masa in $L(\mathbb{F}_n)$ is maximal amenable. They were able to do this by modifying R\u{a}dulescu's proof of singularity of the Laplacian masa to show that it also had the asymptotic orthogonality property.  In 2010, Jolissaint \cite{Jo} gave conditions on a subgroup $H$ of $\Gamma$ that imply that $L(H)$ is a maximal amenable subalgebra of $L(\Gamma)$ by first establishing the asymptotic orthogonality property.  A special case of this showed that $L(H_1)$ is maximal amenable in the group von Neumann algebra of the amalgamated free product group $L(H_1 \ast_Z H_2)$, where $H_1$ is infinite and abelian and $Z$ is a finite common subgroup.  In 2013, Boutonnet and Carderi \cite{BC} expanded Jolissaint's work to show that for any infinite maximal amenable subgroup $H$ of a hyperbolic group $\Gamma$, $L(H)$ is maximal amenable inside $L(\Gamma)$.  Houdayer \cite{Houd3} proved in 2014, among other things, that any diffuse amenable von Neumann algebra can be realized as a maximal amenable subalgebra with expectation inside a full nonamenable type $\text{III}_1$ factor by defining a relative version of the asymptotic orthogonality property, which will be used in this paper.  Other maximal amenable von Neumann algebra results have been obtained by Brothier \cite{Br}, Fang \cite{F}, Gao \cite{Gao}, Ge \cite{Ge}, Hou \cite{Hou}, Houdayer \cite{Houd1}, Houdayer and Ueda \cite{HU}, Shen \cite{Shen}, and Str\u{a}til\u{a} and Zsid\'{o} \cite{SZ}.

More recently, Boutonnet and Carderi \cite{BC2} discovered an alternative technique for establishing maximal amenability in the specific cases of von Neumann algebras arising from groups, and their method was the first result that did not use Popa's asymptotic orthogonality. Boutonnet and Houdayer \cite{BH1}, \cite{BH2} have also used alternative techniques to obtain generalized results in maximal amenability in amalgamated free products.

In this paper, we will utilize Popa's method to obtain a criterion for maximal amenability in amalgamated free product von Neumann algebras.  That is, we will consider finite von Neumann algebras $(N_1, \tau_1)$ and $(N_2, \tau_2)$ with a sufficiently nice common von Neumann subalgebra $(B, \tau_B)$, and we let $M = N_1 \ast_B N_2$ denote their 
amalgamated free product von Neumann algebra.  We assume further that $N_1$ is diffuse and amenable, and that no corner of $N_1$ embeds into a corner of $B$ inside $N_1$, in the sense of Popa's intertwining by bimodules in \cite{P5}.  From this, we will be able to conclude maximal amenability of $N_1$ in $M$, and this is the main theorem of this work. 
The precise statement is as follows. 
\begin{thma} Suppose that $(N_1, \tau_1)$ and $(N_2, \tau_2)$ are finite von Neumann algebras with a common von Neumann subalgebra $(B, \tau_B)$ such that the inclusion $B\subset N_1$ admits a Pimsner-Popa basis of unitaries.  Suppose further that $N_1$ is diffuse and amenable, with the property that $N_1\not\prec_{N_1} B$.  Then $N_1$ is maximal amenable in $M=N_1\ast_B N_2$. \end{thma}
To prove this, we will show that $N_1$ satisfies the relative version of the asymptotic orthogonality property in $M$ using the Pimsner-Popa basis, and that $N_1$ satisfies a relative weak mixing condition in $M$ using the intertwining by bimodules condition.
We also show that the intertwining by bimodules condition $N_1 \not\prec_{N_1} B$ can be translated into a condition on the Pimsner-Popa index of $B$ inside $N_1$, which gives the following corollary.
\begin{corb} Suppose that $(N_1, \tau_1)$ is a finite, diffuse, amenable von Neumann algebra and $(N_2,\tau_2)$ is a finite von Neumann algebra with a common von Neumann subalgebra $(B, \tau_B)$ such that the inclusion $B\subset N_1$ admits a Pimsner-Popa basis of unitaries.  Suppose further that for every non-zero projection $p\in B^\prime\cap N_1$, we have that $[pN_1p : Bp] = \infty$.  Then $N_1$ is maximal amenable in $M = N_1 \ast_B N_2$.
\end{corb} 
The structure of the paper is as follows: in the second section, we present the preliminary material needed for the proofs of the main results.  The third section begins with the construction of the amalgamated free product von Neumann algebra and relevant material on Pimsner-Popa index and bases.  We then prove in Theorem \ref{relAOP} that under our hypotheses, we have the relative asymptotic orthogonality property.  Then combined with the weak mixingness established in Lemma \ref{mix}, we are able to derive maximal amenability and prove Theorem 1.1.  In the fourth section, we give examples of how these hypotheses can be applied in different von Neumann algebra constructions.

\noindent \textbf{Acknowledgment.} I would like to thank Sorin Popa, Cyril Houdayer, and Ben Hayes for their useful comments, suggestions, and discussions regarding the development of this paper.

\section{Preliminaries}
\subsection{Amenability}
We first recall the definition of amenability for von Neumann algebras.
\begin{definition}A von Neumann algebra $M\subset B(\mathcal{H})$ is said to be \emph{amenable} if there exists a state $\Phi$ on $B(\mathcal{H})$ such that $\Phi(uxu^*) = \Phi(x)$ for all $x\in B(\mathcal{H})$ and all $u\in\mathcal{M}$ (that is, $\Phi$ is an $M$-central state on $B(\mathcal{H})$), and such that $\restr{\Phi}{M} = \tau$. We call $\Phi$ a \emph{hypertrace}.
\end{definition}

A von Neumann subalgebra $N$ of $M$ is said to be \emph{maximal amenable} if $N$ is amenable and $N$ is maximal with respect to inclusion among all amenable subalgebras of $M$.  We note that  maximal amenability is, in some sense, a strong type of singularity, as every maximal amenable von Neumann subalgebra of a finite von Neumann algebra is singular (see Lemma 3.6 in \cite{OP}, for example).
\subsection{Asymptotic Orthogonality}
Popa's method relies on the following property:
\begin{definition} A von Neumann subalgebra $N\subset M$ is said to have the \emph{asymptotic orthogonality property} if for any free ultrafilter $\omega$ on $\mathbb{N}$, for any $y_1, y_2\in M \ominus N$, and for any $x\in N^\prime \cap (M^\omega \ominus N^\omega)$, we have that $xy_1 \perp y_2 x$ in $L^2(M^\omega)$.\\
In particular, for any $x\in N^\prime\cap M^\omega$, $y_1, y_2\in M\ominus N$, 
\[
\left\| y_1x - xy_2\right\|_2^2 \geq \left\| y_1 (x-E_{N^\omega}(x))\right\|_2^2 + \left\| ( x-E_{N^\omega}(x))y_2\right\|_2^2.
\]
\end{definition}
Popa proved that any strongly mixing masa satisfying the asymptotic orthogonality property in a finite von Neumann algebra is necessarily maximal amenable. 
More specifically, he showed that if $a\in\mathbb{F}_n$ is one of the $n$ canonical generators, and $M_a$ denotes the abelian subalgebra generated in $L(\mathbb{F}_n)$ by the unitary $u_a$, then $M_a\in L(\mathbb{F}_n)$ has the asymptotic orthogonality property, which he proves in Lemma 2.1 in \cite{P1}.  

In \cite{Houd2}, Houdayer defined the following relative version of Popa's asymptotic orthogonality property:
\begin{definition} For von Neumann algebras $P\subset N \subset M$, we say that the inclusion $N\subset M$ has the \emph{asymptotic orthogonality property relative to $P$} if for any free ultrafilter $\omega$ on $\mathbb{N}$, for any $y_1, y_2\in M\ominus N$, and for any $x\in P^\prime \cap (M^\omega \ominus N^\omega)$, we have that $xy_1 \perp y_2x$ in $L^2(M^\omega)$.
\end{definition}
\subsection{Jones' Basic Construction}
If $Q\subset (M, \tau)$ is a von Neumann subalgebra of a finite von Neumann algebra, then we define the Jones projection $e_Q\in B(L^2(M))$ to be the orthogonal projection from $L^2(M)$ onto $L^2(Q)$.  If $E_Q: M \rightarrow Q$ is the unique trace-preserving faithful normal conditional expectation, then we have that $E_Q$ is implemented on $M$ by $e_Q$ via $e_Qxe_Q = E_Q(x)e_Q$ for all $x\in M$.  We define the basic construction $\left\langle M, e_Q\right\rangle$ to be the von Neumann subalgebra of $B(L^2(M))$ generated by $M$ and $e_Q$.  Then $\left\langle M, e_Q \right\rangle$ has a trace given by $\text{Tr}(xe_Qy) = \tau(xy)$ for any $x, y\in M$.  If we define $J\in B(L^2(M))$ by $x\hat{1} \mapsto x^*\hat{1}$ to be the canonical anti-unitary operator in $B(L^2(M))$, then we have that $\left\langle M, e_Q\right\rangle = JQ^\prime J \cap B(L^2(M))$.  For more details on the basic construction and properties, see \cite{JS}.
\subsection{Popa's Intertwining by Bimodules}
The following definition and theorem were given by Popa. (See Theorem 2.1 in \cite{P5} and the appendix in \cite{P6}.)
\begin{definition} Let $M$ be a finite von Neumann algebra and let $P, Q$ be von Neumann subalgebras of $M$.  We say that \emph{a corner of $P$ embeds into $Q$ inside $M$}, and write $P\prec_M Q$ if there exist non-zero projections $p\in P$ and $q\in Q$, a unital $\ast$-homomorphism $\Psi: pPp \rightarrow qQq$, and a partial isometry $v\in M$ such that $vv^*\in (pPp)^\prime\cap pMp$, $v^*v\in (\Psi(pPp))^\prime\cap qMq$, and such that $xv = v\Psi(x)$ for all $x\in pPp$, with $x=0$ whenever $xv = 0$.
\end{definition}
\begin{theorem}(Popa) For $M$ a finite von Neumann algebra with von Neumann subalgebras $P$ and $Q$, the following are equivalent:
\begin{enumerate}
\item $P\prec_M Q$.
\item There exists a Hilbert space $\mathcal{H}\subset L^2(M)$ such that $\mathcal{H}$ is a $P$-$Q$-bimodule with finite dimension as a right $Q$-module.
\item There exists a nonzero projection $f_0\in P^\prime \cap \left\langle M, e_Q \right\rangle$ with finite trace.
\item It is not true that: For any $a_1, a_2, \ldots, a_n\in M$ and any $\varepsilon>0$, there exists a unitary $u\in P$ such that $\left\|E_Q(a_iua_j^*)\right\|_2 < \varepsilon$ for each $i,j$.
\end{enumerate}
\end{theorem}
\subsection{Pimsner-Popa Index and Bases}
We recall some definitions and facts from \cite{PiPo1}.  The Pimsner-Popa index for an inclusion of finite von Neumann algebras was first considered as a generalization of the Jones index for an inclusion of subfactors. 
\begin{definition}
For an inclusion of finite von Neumann algebras $P\subset N\subset M$, we define the constant $\lambda(N, P) = \max\{\lambda\geq 0 : E_{P}(x)\geq \lambda x \text{ for all } x\in N \text{ with } x>0\}$.\\
  Then the index can be defined by $[N : P]^{-1} = \lambda(N, P)$.
\end{definition}
We will need to consider the case when the Pimsner-Popa index is infinite, and we will use the following lemma from \cite{PiPo1}.
\begin{lemma}  Let $M$ be a von Neumann algebra with a faithful, normal, semifinite trace $\tau$, and suppose that $N\subset M$ is a von Neumann subalgebra such that $N^\prime\cap M$ contains no finite trace projections of $M$.  Then for all $\varepsilon > 0$ and all $x\in M$ with $\tau(x) <\infty$, there exist projections $q_1, q_2, \ldots, q_n\in N$ such that $\sum_{i=1}^n q_1 = 1$ and $\left\| \sum q_i x q_i \right\|_2 < \varepsilon \left\| x \right\|_2$.
\end{lemma}

We then use this lemma to establish the following result.
\begin{lemma} Suppose that $M$ is a diffuse amenable finite von Neumann algebra, and $N\subset M$ is a von Neumann subalgebra.  The following are equivalent:
\begin{enumerate}
\item $M \not\prec_M N$.
\item $[pMp : Np] = \infty$ for every non-zero projection $p\in N^\prime\cap M$.
\item There exist unitaries $u_n\in \mathcal{U}(M)$ such that $\left\| E_N(u_n x) \right\|_2 \rightarrow 0$ for every $x\in N$.
\end{enumerate}
\end{lemma}
\begin{proof} ($1\Rightarrow 2$) (c.f. Lemma 2.4 in \cite{CIK}, Lemma 1.4 in \cite{I1}).  Suppose that $M\not\prec_M N$ and let $p\in N^\prime\cap M$ be a non-zero projection.  It follows that $pMp\not\prec_M Np$.  Then by the intertwining by bimodules theorem, we have that $(pMp)^\prime \cap \left\langle M, e_{Np}\right\rangle$ does not contain a non-trivial projection of finite trace.  Thus, by the above lemma from Pimsner-Popa, we have that for any $\varepsilon > 0$, there exist projections $q_1, q_2\ldots q_n \in pMp$ such that $\sum_{i=1}^n q_i = 1$ and $\left\| \sum_{i=1}^n q_i e_{Np} q_i\right\|_{2,\text{Tr}} < \varepsilon$.  But then $\left\| \sum_{i=1}^n q_i e_{Np} q_i\right\|_{2,\text{Tr}}^2 = \sum_{i=1}^n \left\| E_{Np}(q_i)\right\|_2^2$, so there exists some $q = q_i$ with $\left\| E_{Np}(q)\right\|_2 \leq \varepsilon \left\| q\right\|_2$.  Thus, $[pMp : Np] = \infty$.\\
\\
($2\Rightarrow 3$)  Let $x_1, \ldots x_k\in N$ and $\varepsilon >0$.  We define $\xi = \sum x_j \otimes x_j^{\circ}\in L^2(M)\otimes_N L^2(M)$.  Thus, for any $u\in \mathcal{U}(M)$, we have $\left\langle u^* \xi u, 1\otimes 1\right\rangle = \sum_j \left\| E_N(ux_j)\right\|_2^2$.  But if $\sum \left\| E_N(ux_i)\right\|_2^2 \geq \varepsilon$, there exists $\xi\in L^2(M)\otimes_N L^2(M)$ that commutes with $M$ with $\xi\neq 0$.  But $L^2(M)\otimes_N L^2(M) \cong L^2(\left\langle M, e_N\right\rangle)$, and $\left\langle M, e_N\right\rangle \cap M^\prime = \rho(M\cap N^\prime)$, so by letting $p$ be a spectral projection of $\left| \xi \right|$ with finite trace, we have that $p\in N^\prime \cap M$ with $[pMp:Np] <\infty$.\\
\\
($3\Rightarrow 1$) This follows from the intertwining by bimodules theorem.
\end{proof}

Also from \cite{PiPo1}, we are concerned with an inclusion $N\subset M$ of finite von Neumann algebras with the following property:
\begin{definition} Suppose $N\subset M$ are finite von Neumann algebras and suppose there exists a family $\{m_j\}_{j\in J}$ of elements in $M$ such that:\\
(a) $E_N(m_j^*m_k) = 0$ if $j\neq k$,\\
(b) $E_N(m_j^*m_j)$ is a projection in $N$, and\\
(c) $L^2(M, \tau) = \bigoplus_j m_jL^2(N, \tau)$, and every $x\in M$ has a decomposition $x=\sum_j m_jE_N(m_j^*x)$.\\
Then the family $\{m_j\}_{j\in J}$ is called a \emph{Pimsner-Popa basis of $M$ over $N$}.
\end{definition}
This property was introduced and studied by Pimsner and Popa in \cite{PiPo1}, and they showed that the index cannot be greater than the size of the Pimsner-Popa basis.  Thus, we shall consider inclusions with infinite Pimsner-Popa bases.  Moreover, we will be focusing on the case when the Pimsner-Popa basis consists of unitary elements of $M$, which was studied by Ceccherini-Silberstein in \cite{CS}, where he introduced the terminology that the inclusion $N\subset M$ has the \emph{$U$-property} if there exists a Pimsner-Popa basis of $M$ over $N$ consisting of unitaries.

\section{Main Result}
\subsection{Amalgamated Free Products}  We now recall the construction of amalgamated free products as introduced in \cite{V} and \cite{P2}.  Suppose that $(N_1, \tau_1)$ and $(N_2, \tau_2)$ are finite von Neumann algebras with a common von Neumann subalgebra $(B, \tau_B)$, where $\restr{\tau_j}{B} = \tau_B$ for each $j$, so there exists a trace preserving conditional expectation $E_j: N_j \rightarrow B$.  We define the \emph{free product of $N_1$ and $N_2$ with amalgamation over $B$}, denoted $(M, E_B) = (N_1, E_1) \ast_B (N_2, E_2) = N_1 \ast_B N_2$, as follows:\\
\medskip
$N_1 \ast_B N_2$ has a dense *-subalgebra
\begin{equation*}
B \oplus \bigoplus_{n\geq  1} \bigoplus_{\substack{i_j\in\{1,2\}\\
		i_1\neq i_2, i_2\neq i_3, \ldots, i_{n-1}\neq i_n}} \text{sp}(N_{i_1}\ominus B)(N_{i_2}\ominus B) \cdots (N_{i_n} \ominus B)
\end{equation*}
with $E_B: M\rightarrow B$ satisfying $\restr{E_B}{N_j} = E_j$ and $E_B(x) = 0$ for any word $x = x_{i_1}x_{i_2}\ldots x_{i_n}$ with $x_{i_k}\in N_{i_k} \ominus B$, $i_1\neq i_2 \neq \ldots \neq i_n$.

$M$ has a trace given by $\tau = \tau_B \circ E_B$, and the subspaces $B$ and $\text{sp}(N_{i_1}\ominus B)\cdots (N_{i_n} \ominus B)$ are all mutually orthogonal in $L^2(M, \tau)$ with respect to the inner product $\left\langle x, y \right\rangle = \tau(y^*x)$, and the closures in $L^2(M, \tau)$ of these subspaces give mutually orthogonal Hilbert $B$-$B$-bimodules that sum to $L^2(M, \tau)$:
\[
L^2(\text{sp}(N_{i_1}\ominus B)\cdots (N_{i_n}\ominus B)) \cong \mathcal{H}_{i_1}^0 \otimes_B \mathcal{H}_{i_2}^0 \otimes_B \cdots \otimes_B \mathcal{H}_{i_n}^0,
\]
with $\mathcal{H}_i^0 = L^2(N_i) \ominus L^2(B)$.
\subsection{Maximal Amenability}
In this section, we will establish the results about amalgamated free products needed to prove the main theorem.

As far back as \cite{P4}, mixing properties were used to study subalgebras of $\text{II}_1$ factors. We recall the following definition as stated in \cite{PV}.
\begin{definition} For finite von Neumann algebras $P\subset N \subset M$, we say that the inclusion $N\subset M$ is \emph{weakly mixing through $P$} if there exists a sequence $u_n \in \mathcal{U}(P)$ such that
\[
\left\| E_N (xu_ny)\right\|_2 \rightarrow 0 \text{ for all } x,y\in M\ominus N.
\]
\end{definition}
The intertwining by bimodules condition ensures this weak mixingness in amalgamated free products, as can be seen from the following result in \cite{IPP}, which is contained in their proof of Theorem 1.1.
\begin{lemma}\label{mix} Suppose $(N_1, \tau_1)$ and $(N_2, \tau_2)$ are finite von Neumann algebras with a common subalgebra $B$ such that $\restr{\tau_1}{B} = \restr{\tau_2}{B}$.  Let $M =N_1\ast_B N_2$, and suppose that $P\subset N_1$ is a von Neumann subalgebra such that $P\not\prec_{N_1} B$.  Then for each $k=1,2$, $N_k\subset M$ is weakly mixing through $P$. 
\end{lemma}
In order to get maximal amenability, we first show that $N_1\subset M$ has the asymptotic orthogonality property, using in part a technique inspired by \cite{Houd1} and \cite{HR}.
\begin{theorem}\label{relAOP} Suppose $(N_1, \tau_1)$ and $(N_2, \tau_2)$ are finite von Neumann algebras with a common subalgebra $B$ such that $\restr{\tau_1}{B} = \restr{\tau_2}{B}$.  Suppose further that $B\subset N_1$ admits a Pimsner-Popa basis of unitaries $\{u_i\}_{i\in I}$ . Let $M =N_1\ast_B N_2$.  Let $\omega$ be a free ultrafilter on $\mathbb{N}$, $y_1, y_2\in M \ominus N_1$, and $x\in N_1^\prime \cap (M^\omega \ominus N_1^\omega)$. Then $xy_1 \perp y_2x$ in $L^2(M^\omega, \tau_\omega)$.
\end{theorem}
\begin{proof}
First, we note that if we let $(x_n)_n$ be a sequence representing $x\in M^\omega$, it suffices to show that for every $\varepsilon>0$ there is an integer $N$ such that $\left\langle x_ny_1, y_2 x_n\right\rangle < \varepsilon$ for every $n\geq N$.\\
Secondly, by Kaplansky density, it suffices to consider the case when $y_1$ and $y_2$ are monomials.  We write $y_1 = a_1 \cdots a_{2k+1}$ and $y_2 = b_1\cdots b_{2l+1}$, for $a_1, a_{2k+1}, b_1, b_{2l+1}\in (N_1\ominus B)\cup\{1\}$, $a_{2i}, b_{2j}\in N_2\ominus B$ for all $1\leq i\leq k$, $1\leq j\leq l$, and $a_{2i+1}, b_{2j+1}\in N_1\ominus B$ for all $1\leq i \leq k-1$, $1\leq j \leq l-1$.\\

We define $X_1$ to be the closure in $L^2(M)$ of the subspace consisting of sums of reduced words in $M$ that contain at least one letter from $N_2\ominus B$ and that begin with some letter $c\in (N_1\ominus B)\cup\{1\}$ and end with a letter $d\in (N_1\ominus B)\cup\{1\}$ such that $E_B(b_{2l+1}c) = E_B(b_1^*c) = 0$ if $c\neq1$,  and $E_B(da_1) = 0$ if $d\neq 1$.\\

We define $X_0 = L^2(M)\ominus X_1$. Let $(x_n)_{X}$ denote the projection of $x_n$ onto a subspace $X$.  Then we decompose $x_n = (x_n)_{X_0} + (x_n)_{X_1}$ for every $n$.  We claim that $(x_n)_{X_1}$ satisfies the mutual orthogonality with $y_1$ and $y_2$, and we claim that $\lim_{n\rightarrow \omega}\left\|(x_n)_{X_0}\right\|_2=0$.\\

We use the Pimsner-Popa basis of unitaries $\{u_i\}_{i\in I}$ to establish the latter claim.  
Thus, we have that $\left\langle u_i X_0 u_i^*, u_j X_0 u_j^* \right\rangle = 0$ whenever $i\neq j$, and hence, for each $k$, we have
\begin{align*}
\left\| (x_n)_{X_0} \right\|_2^2 &= \left\| u_k (x_n)_{X_0} u_k^*\right\|_2^2\\
&\leq 2\left\| u_k (x_n)_{X_0} u_k^* - (x_n)_{u_k X_0 u_k^*}\right\|_2^2 + 2\left\| (x_n)_{u_k X_0 u_k^*}\right\|_2^2\\
&\leq 2\left\| u_kx_nu_k^* - x_n\right\|_2^2 + 2\left\| (x_n)_{u_k X_0 u_k^*}\right\|_2^2.
\end{align*}
By summing over $k$, we have that 
\[
\sum_{k=1}^{m} \left\| (x_n)_{X_0}\right\|_2^2 \leq 2\sum_{k=1}^{m}\left\|u_k x_n u_k^* - x_n \right\|_2^2 + 2\sum_{k=1}^{m} \left\| (x_n)_{u_k X_0 u_k^*}\right\|_2^2.
\]
But by orthogonality of the Pimsner-Popa basis, 
\begin{align*}
\sum_{k=1}^{m} \left\| (x_n)_{u_k X_0 u_k^*}\right\|_2^2 &\leq \left\| (x_n)_{u_1X_0u_1^*\cup u_2X_0u_2^*\cup\cdots \cup u_{m}X_0 u_{m}^*}\right\|_2^2\\
&\leq  \left\| x_n \right\|_2^2.
\end{align*}
Thus, we have that 
\[
m \left\|(x_n)_{X_0}\right\|_2^2 \leq  2\sum_{k=1}^{m}\left\|u_k x_n u_k^* - x_n \right\|_2^2 + 2\left\| x_n \right\|_2^2,
\]
so that
\[
\left\|(x_n)_{X_0}\right\|_2^2 \leq \frac{2}{m} \sum_{k=1}^{m}\left\|u_k x_n u_k^* - x_n \right\|_2^2 + \frac{2}{m} \left\| x_n \right\|_2^2
\]
Since $x\in N_1^\prime \cap (M^\omega \ominus N_1^\omega)$, we have that 
\[
\lim_{n\rightarrow \omega} \left\|(x_n)_{X_0}\right\|_2^2 \leq \frac{2}{m} \lim_{n\rightarrow \omega} \left\|x_n\right\|_2^2.
\]
As this is true for every $m>0$, we have that $\lim_{n\rightarrow\omega} \left\| (x_n)_{X_0}\right\|_2 = 0$.\\

Thus, it suffices to show that $(x_n)_{X_1}$ satisfies the mutual orthogonality property. 
Namely, $\left\langle (x_n)_{X_1}y_1, y_2 (x_n)_{X_1}\right\rangle =  \tau((x_n)_{X_1}^*y_2^*(x_n)_{X_1}y_1)$.\\
But since $(x_n)_{X_1} \in X_1$, we have that 
\[
E_B( (x_n)_{X_1}^*b_{2l+1}^*)=0\text{, }E_B(b_1^* (x_n)_{X_1}) = 0\text{, and } E_B((x_n)_{X_1} a_1) = 0.
\]
Hence, $\left\langle (x_n)_{X_1}y_1, y_2 (x_n)_{X_1}\right\rangle = 0$, so $\lim_{n\rightarrow\omega} \left\langle x_n y_1, y_2 x_n\right\rangle = 0$.

This completes the proof.
\end{proof}

As we have established asymptotic orthogonality, as in Popa's result, we use weak mixingness to get maximal amenability.  To be precise, we use the following generalization of Popa's result given by Houdayer:
\begin{theorem}\label{HAOP}(Theorem 8.1 in \cite{Houd2}) Suppose $P\subset N \subset M$ are tracial von Neumann algebras.  Assume the following:
\begin{enumerate}
\item $P$ is amenable.
\item $N \subset M$ is weakly mixing through $P$.
\item $N \subset M$ has the asymptotic orthogonality property relative to $P$.
\end{enumerate}
Then for any intermediate amenable von Neumann subalgebra $P\subset Q \subset M$, we have that $Q\subset N$.
\end{theorem}

Now our theorem follows easily, as we have established the weak mixing condition and the relative asymptotic orthogonality.
\begin{thma} Suppose that $(N_1, \tau_1)$ and $(N_2, \tau_2)$ are finite von Neumann algebras with a common von Neumann subalgebra $(B, \tau_B)$, where $N_1$ has a Pimsner-Popa basis of unitaries over $B$.  Suppose that $N_1$ is diffuse and amenable and that $N_1\not\prec_{N_1} B$.  Let $M=N_1\ast_B N_2$.  Then $N_1$ is maximal amenable in $M$.
\end{thma}
\begin{proof} Assume for sake of contradiction that there exists an intermediate amenable von Neumann algebra, $N_1\subsetneq Q\subset M$.  Then by Proposition \ref{mix}, $N_1\subset M$ is weakly mixing, and by Proposition \ref{relAOP}, $N_1\subset M$ has the asymptotic orthogonality property.  Hence, by Theorem \ref{HAOP}, we have that $Q\subset N_1$.  Thus, $Q=N_1$, a contradiction.
\end{proof}

\section{Examples}

\subsection{Group von Neumann algebra case} Consider the case where $\Lambda_1$ is a discrete countable amenable group and $H$ is a subgroup of $\Lambda_1$ with $[\Lambda_1 : H] = \infty$.  Note that the group unitaries form a Pimsner-Popa basis in this case. Let $\Lambda_2$ be another group containing $H$.  Define $\Gamma$ to be the amalgamated free product group $\Gamma = \Lambda_1 \ast_H \Lambda_2$.  Then $L(\Lambda_1)$ is a maximal amenable von Neumann subalgebra of $L(\Gamma) = L(\Lambda_1)\ast_{L(H)}L(\Lambda_2)$.  Similarly, if $\Lambda_1\curvearrowright X$ is a free, ergodic, probability measure preserving action on a probability space $X$, then if we take $N_1 = L^\infty(X)\rtimes \Lambda_1$ and $B = L^\infty(X)\rtimes H$, we also get that $N_1$ is maximal amenable in $N_1\ast_B N_2$ for any $N_2$ contained $B$.  This had been previously proven by Jolissaint in \cite{Jo} in the case where $H$ was assumed to be finite.  The infinite index case was proven independently by Boutonnet and Carderi in \cite{BC2} as a corollary of their more general result.  

\subsection{Abelian case} For the abelian case, we have $B = L^\infty(Y, \nu) \subset L^\infty(X, \mu) = N_1$.  This inclusion induces a surjective measure-preserving map $\phi: X\rightarrow Y$.  We can disintegrate the measure space as a direct integral, 
\[
(X, \mu) = \int_Y^{\oplus} (\phi^{-1}(\{ y\}), \mu_y) \text{d}\nu(y),
\]
where $\mu_y$ is a probability measure on $\phi^{-1}(\{y\})$ such that:
\begin{itemize}
\item For every measurable $E\subset X$, $y\mapsto \mu_y(E\cap \phi^{-1}(\{y\}))$ is measurable
\item $\mu(E) = \int_Y \mu_y(E\cap \phi^{-1}(\{y\}))\text{d}\nu(y).$
\end{itemize}
Then the inclusion having the desired intertwining by bimodules condition is equivalent to the following measure theoretic condition.
\begin{proposition} $[L^\infty(X, \mu)p : L^\infty(Y, \nu)p] =\infty$ for all $p\in L^\infty(Y,\nu)^\prime\cap L^\infty(X,\mu)$ if and only if $\mu_y$ is a diffuse probability measure for a.e. $y\in Y$.
\end{proposition}
\begin{proof}
Suppose a positive measure set in $Y$ has fibers with atoms.  Then we can define the set $E = \{y: \mu_y \text{ has an atom}\}\subset Y$.  Then $E$ is measurable and $\phi: \phi^{-1}(E)\rightarrow E$ is surjective, so by measurable selection, there exists a map $\psi: E\rightarrow \phi^{-1}(E)$ such that $\phi\circ \psi = \text{id}_E$ and $\psi(y)$ is an atom of $\mu_y$ for every $y\in F$.  Let $p=\chi_{\psi(E)}$.  We claim that $[L^\infty(X, \mu)p : L^\infty(Y, \nu)p] = 1$.  To see this, let $f\in L^\infty(X, \mu)$.  Note that $f(x)\chi_{\psi(E)}(x) \neq 0$ implies that there exists $y\in E$ such that $\psi(y) = x$, so $x\in \phi^{-1}(\{y\})$ is an atom of $\mu_y$.  Thus, $f(x)\chi_{\psi(E)}(x) = f(\psi(y))\chi_E(y)$.\\
We define $g:E\rightarrow \mathbb{C}$ by $g(y) = f(\psi(y))$.  Then for every $x\in \psi(E)$, we have that $x=\psi(y)$ for some $y$, so 
\[
g(\phi(x)) = g(y) = f(\psi(y)) = f(x).
\]
Thus, $(g\circ \phi)\chi_{\psi(E)} = f\chi_{\psi(E)}$, so $L^\infty(Y, \nu)p = L^\infty(X, \mu)p$.\\
\\
For the other direction, first assume that $[L^\infty(X, \mu): L^\infty(Y, \nu)] <\infty$.  As noted in the proof of Prop 1.2 in \cite{P3}, we have that 
\[
[L^\infty(X, \mu): L^\infty(Y, \nu)]^{-1} = \inf\{ \mathcal{E}(p)(\phi(x)) : x\in X, p\in L^\infty(X, \mu) \text{ a projection, } p(x)\neq 0\},
\]
where $\mathcal{E} : L^\infty(X, \mu) \rightarrow L^\infty(Y, \nu)$ is the conditional expectation.
Thus, there exists $\varepsilon>0$ such that $\mathcal{E}(p)(y) \geq \varepsilon$ for every $y\in Y$ and $p\in L^\infty(X, \mu)$ such that $p(\phi^{-1}(\{y\}))\neq 0$.  Using the disintegration of measure, we have that for every $f\in L^\infty(X, \mu)$, 
\[
\int_X f d\mu = \int_Y \left( \int_{\phi^{-1}(\{y\})}f\>d\mu_y\right) d\nu(y).
\]
Hence, $\mathcal{E}$ is defined by $\mathcal{E}(f)(y) = \int_{\phi^{-1}(\{y\})} f d\mu_y$.  Therefore, if $E\subset X$ is any measurable subset, we have that $\mathcal{E}(\chi_E)(y) = \mu_y(E\cap \phi^{-1}(\{y\}))$.\\
Therefore, for every $y\in Y$ and for every measurable $E\subset X$ such that $E\cap \phi^{-1}(\{y\}) \neq \emptyset$, we have that $\mu_y(E\cap \phi^{-1}(\{y\})) \geq \varepsilon$.  Hence, $\mu_y$ has atoms for a.e. $y$.\\
The argument remains the same if we assume instead that $[L^\infty(X, \mu)p : L^\infty(Y,\nu)p]<\infty$ for some $p\in L^\infty(Y,\nu)^\prime \cap L^\infty(X, \mu)$, so the proof is complete.\\
\end{proof}

\end{document}